\documentclass{amsart}

\usepackage{latexsym,amsmath,amssymb,amsthm,amsfonts}

\usepackage{enumerate}

\usepackage[numeric]{amsrefs}

\newcommand{\qtq}[1]{{\quad\text{#1}\quad}}

\theoremstyle{plain} 

\newtheorem{thm}{Theorem}[section]
\newtheorem{cor}[thm]{Corollary}
\newtheorem{lem}[thm]{Lemma}
\newtheorem{prop}[thm]{Proposition}

\newtheorem{defn}[thm]{Definition}

\theoremstyle{remark}
\newtheorem{rem}[thm]{Remark}

\numberwithin{equation}{section}

\def\f{\frac}

\def\vi{\varphi}
\def\({\left(}
\def \){ \right)}

\def\Bl{\Bigl}
\def\Br{\Bigr}
 
 \def\ee{{\textnormal{e}}}
 
 \def\tr{{\triangle}}
 
\def\ta{\theta}

\def\da{{\delta}}
\def\sa{{\sigma}}

 \def\va{\varepsilon}

\def\EEE{{\mathcal E}}

 \def\NN{{\mathbb N}}

 \def\RR{{\mathbb R}}
 \def\SS{{\mathbb S}}

  \def\dim{\operatorname{dim}}

  \def\sph{\mathbb{S}^{d-1}}
  
\def\Og{\Omega}

\newcommand{\wt}{\widetilde}
\newcommand{\wh}{\widehat}

\newcommand{\spn}{{\rm span}}

\def\p{\partial}

\def\og{\omega}

\newcommand{\R}{{\mathbb{R}}}
\newcommand{\eps}{\varepsilon}

\newcommand{\diam}{{\rm diam}}

\def\be{\begin{equation}}
\def\ee{\end{equation}}
\def\CS{\mathcal{S}}

 \def\CE{\mathcal{E}}
 \def\cN{\mathcal{N}}

\begin{document}

\title[]{On directional Whitney inequality}
\author{Feng Dai}
\address{Department of Mathematical and Statistical Sciences\\
	University of Alberta\\ Edmonton, Alberta T6G 2G1, Canada.}
\email{fdai@ualberta.ca}

\author{Andriy Prymak}
\address{Department of Mathematics, University of Manitoba, Winnipeg, MB, R3T2N2, Canada}

\email{prymak@gmail.com}

\thanks{	The first author was supported by  NSERC of Canada Discovery
	grant RGPIN-2020-03909, and the second author  was supported by NSERC of Canada Discovery grant RGPIN-2020-05357.
	}


\keywords{Whitney-type inequality, directional modulus of smoothness, convex domains, $C^2$-domains, multivariate polynomials, illumination of convex bodies, X-ray number, illumination number}
\subjclass[2020]{41A10, 41A25, 41A63, 52A20, 52A40}

\begin{abstract}
This paper  studies  a  new   Whitney type inequality on a compact domain $\Og\subset \RR^d$   that takes the form
$$\inf_{Q\in \Pi_{r-1}^d(\CE)} \|f-Q\|_p \leq C(p,r,\Og) \og_{\CE}^r(f,\diam(\Og))_p,\  \ r\in \NN,\  \ 0<p\leq \infty,$$
where    $\og_{\CE}^r(f, t)_p$ denotes  the $r$-th order directional modulus of smoothness of $f\in L^p(\Og)$  along   a finite  set  of directions $\CE\subset \sph$ such that $\spn(\CE)=\RR^d$,  $\Pi_{r-1}^d(\CE):=\{g\in C(\Og):\  \og^r_\CE (g, \diam (\Og))_p=0\}$.
 We prove  that  there does not exist  a universal   finite set  of directions $\CE$ for which this   inequality holds on  every convex body $\Og\subset \RR^d$, but     for every  connected $C^2$-domain $\Og\subset \RR^d$, one can choose $\CE$ to be   an arbitrary  set   of $d$ independent directions.
  We also study the smallest number $\cN_d(\Og)\in\NN$ for which there exists a set  of $\cN_d(\Og)$ directions $\CE$  such that $\spn(\CE)=\RR^d$ and  the directional Whitney inequality holds on $\Og$ for all $r\in\NN$ and $p>0$.
It is proved that $\cN_d(\Og)=d$ for every connected $C^2$-domain $\Og\subset \RR^d$, for $d=2$ and every planar convex body $\Og\subset \RR^2$, and  for $d\ge 3$ and every almost smooth convex body $\Og\subset \RR^d$.  For $d\ge 3$ and a more general convex body $\Og\subset \RR^d$,  we  connect $\cN_d(\Og)$  with a problem in convex geometry on   the X-ray number of $\Og$, proving  that if  $\Og$ is  X-rayed by a finite set  of directions $\CE\subset \sph$, then $\CE$ admits  the directional Whitney inequality  on $\Og$ for all $r\in\NN$ and $0<p\leq \infty$. Such a connection
  allows us to deduce certain  quantitative  estimate of $\cN_d(\Og)$ for $d\ge 3$.

A slight modification of  the proof  of the usual Whitney inequality in literature  also yields  a directional Whitney inequality on each convex body $\Og\subset \RR^d$, but with the set $\CE$ containing more than $(c d)^{d-1}$ directions. In this paper, we develop a new and  simpler  method to prove the directional Whitney inequality  on  more general, possibly non-convex domains requiring  significantly fewer directions  in the directional moduli.

\end{abstract}

\maketitle

%

\section{Introduction}

\subsection{Definitions and notations}

 Let $\sph:=\{x\in\RR^d:  \|x\|=1\}$ denote  the unit sphere of $\RR^d$. Here and throughout the paper, $\|\cdot\|$ denotes the Euclidean norm.  Given $x\in\RR^d$ and $r>0$, let
 $B_r(x):=\{y\in\RR^d:\  \|x-y\|< r\}$ and $B_r[x]:=\overline{B_r(x)}$, where $\overline{\Og}$ denotes the closure of $\Og\subset\RR^d$. A convex body in $\RR^d$ is a compact convex subset of $\RR^d$ with  nonempty interior.
Let  $\Og$ be  a nonempty bounded measurable  set in $\RR^d$.   We denote by  $L^p(\Og)$, $0<p<\infty$ the usual Lebesgue $L^p$-space  on  $\Og$. In the limiting  case we set $L^\infty(\Og)=C(\overline {\Og})$, the space of all continuous functions  on $\overline {\Og}$ equipped with the uniform norm.

The  $r$-th order  directional  modulus of smoothness of $f\in L^p(\Og)$ in the  direction of  $\xi\in\sph$ is defined   by
$$\og_\xi^r(f, t)_{L^p(\Og)}:= \sup_{|u|\leq t} \|\tr_{ u\xi}^r f\|_{L^p(\Og_{r, u\xi})},\   \ t>0,  \  \ 0<p\leq \infty,$$
where $\Og_{r,h}:=\{ x\in\Og:  \  x, x+h, \dots, x+rh\in \Og\}$, and
$$ \tr_h^r f(x) :=\sum_{j=0}^r (-1)^{r+j} \binom{r} {j} f(x+j h ),\  \ x\in \Og_{r,h},\  \ h\in \RR^d\setminus\{0\}. $$
Given a set of  $\mathcal{E}\subset \sph$ of directions,  define
$$ \og_{\CE} ^r (f,  t)_{L^p(\Og)}:=\sup_{\xi\in \mathcal{E}} \og_\xi^r(f, t)_{L^p(\Og)}\   \ \text{and}\   \  \og_{\CE} ^r(f; \Og)_p:=
\og^r_{\CE} (f,  \diam(\Og))_{L^p(\Og)},$$
where $\diam (\Og) :=\sup_{\xi, \eta \in \Og} \|\xi-\eta\|$.
In the case of $\CE=\sph$, we write $
\og^r(f,t)_p:= \og^r_{\sph} (f,t)_p$.

 Let $\Pi_n^d$ denote the space of all real  algebraic polynomials in $d$ variables of total degree at most $n$. For $f\in L^p(\Og)$ and $0<p\leq \infty$, we define
 $$ E_r(f)_{L^p(\Og)}:=\inf\Bl\{ \|f-q\|_p:\   \  q\in \Pi^d_{r-1}\Br\},\   \ r=1,2,\dots. $$
 
\subsection{Whitney inequalities}
 
 The usual    Whitney type  inequality
 deals with  approximation of a function $f$  on a connected compact  domain   $\Og\subset \RR^d$ by polynomials of total degree $<r$, with the error $E_{r}(f)_{L^p(\Og)}$ estimated in terms of the modulus of smoothness $\og^r(f; \diam(\Og))$.  It   takes  the form
 \begin{equation}\label{whitney}
 E_r(f)_{L^p(\Og)} \leq C \og^r(f, \diam(\Og))_p,\   \ \forall f\in L^p(\Og),
 \end{equation}
  and  allows one to obtain  good   approximation of $f$  for each fixed $r$ if  $\diam(\Og)$ is small.
  Note that  then the Whitney inequality \eqref{whitney} is,  in fact, an equivalence
  since   $\og^r (Q; q)_p=0$ for each  $q\in\Pi^{d}_{r-1}$,  and
 \[
 \og^r (f; \Og)_p=\inf_{q\in\Pi_{r-1}^d} \og^r (f-q; \Og)_p\le C_{r,p} E_{r}  (f)_{L^p(\Og)}.
 \]

 In many  applications,  it  is  very important to have certain  quantitative  estimates of   the smallest constant $C$ for which   \eqref{whitney}  holds. Such a constant is  called the   Whitney  constant and  is denoted by $w_r(\Og)_p$. Thus,
 \begin{equation}
 w_r(\Og)_p:= \sup\Bl \{ E_r(f)_{L^p(\Og)}:\  f\in L^p(\Og) \ \text{and}\   \ \og^r(f; \Og)_p\leq 1\Br\}.
 \end{equation}
 It can be easily seen that for each  non-singular affine transformation $T: \RR^d\to\RR^d$,
\begin{equation}
w_r(\Og)_p=w_r (T(\Og))_p.
\end{equation}

 The Whitney inequality  had been studied extensively in various settings  in literature.
  In the case of one variable, the classical  Whitney theorem asserts  that  $w_r([0,1])_p \leq C_{p,r}<\infty$ for all $r\in\NN$ and $0<p\leq \infty$.  This result was first proved by   Whitney \cite{Wh} for $p=\infty$, extended   by   Yu.~Brudnyi \cite{Bru} to   $1\leq p<\infty$, and by Storozhenko \cite{Stor} to $0<p<1$.
   For $p=\infty$,  Yu.~Brudnyi \cite{Bru} also  proved the estimate $w_r([0,1])_\infty \leq C r^{2r}$ for all $r\in \NN$. This estimate was subsequently improved by a research team (K. Ivanov, Binev, and Takev) headed by Sendov~\cite{Sen}  who finally showed that $w_r([0,1])_\infty \leq 6$   (see~\cite{IT, GKS, Sen} and the references therein).  The best known result in this direction is due to Gilewicz, Kryakin and
  Shevchuk \cite{GKS}, who proved that $w_r([0,1])_\infty \leq 2+e^{-2}$ for all $r\in\NN$.

 A multivariate  extension of the Whitney inequality,  proved by Yu.~Brudnyi \cite{Bru2} for $p\geq 1$ and by Dekel and Leviatan \cite{De-Le} for the full range of $0<p\leq \infty$,  asserts that
  $ w_r(\Og)_p\leq C(p, d, r)<\infty$ for all $r\in \NN$ and every convex body  $\Og\subset \RR^d$. A similar estimate with the constant depending on certain geometric parameters of the domain $\Og$,   was obtained earlier  by Yu.~Brudnyi \cite{Bru2} for $p\ge 1$.
      In the case of  $p=\infty$,  Yu.~Brudnyi and Kalton \cite{Bru-Kalton} obtained some  sharp   estimates of the Whitney constant $w_r(\Og)_\infty$ for  a convex body $\Og\subset \RR^d$    in terms of the dimension $d$.
 For more results on multivariate  Whitney inequalities, we refer to  \cite{Bre-Scot,Di96, DU, BPR}.

\subsection{Directional Whitney inequalities}

The   main purpose  in this paper  is to establish  a proper generalization of the   Whitney   inequality \eqref{whitney}  for  directional  moduli of smoothness along a finite number of directions  on a possibly non-convex domain $\Og$.  Such a generalization is motivated by a recent work of  the current authors
 in an upcoming paper, where we establish both  a direct Jackson inequality and its matching inverse for  a  new  directional modulus of smoothness on a  compact $C^2$-domain, and the usual   Whitney   inequality  is  not  enough for our purposes.

The   generalization requires approximation of  functions  from a  space $\Pi^{d}_{r-1}(\mathcal{E})$ that is larger than $\Pi_{r-1}^d$ and depends  on the set $\CE\subset \sph$ of selected directions.
To see this, assume that $\Og$ is the closure of a connected open set.  Then   a function $f\in L^p(\Og)$ belongs to  the space $\Pi_{r-1}^d$ if and only if  $\og^r(f; \Og)_p =0$. However,  for   directional moduli,  the equality  $\og_{\CE}^r (f;  \Og)_p=0$ may hold   for functions $f\notin \Pi_{r-1}^d$, and therefore,  the inequality
$$E_r(f)_{L^p(\Og)} \leq C \og^r_{\CE}(f; \Og)_p$$
is in general not correct.
 For example, if  $e_1:=(1,0,\dots, 0)$, $\dots, e_{d} =(0,\dots, 0,1)\in\RR^{d}$ and
$\CE=\{e_1,\dots, e_d\}$, then $\og^r(f; \Og)_p=0$ if and only if $f$ is a polynomial of degree $<r$ in each variable.

\begin{defn}\label{def-1-1}
	Given  a set of directions $\mathcal{E}\subset\SS^{d-1}$,  we define $\Pi^{d}_{r-1}(\mathcal{E})$  to be  the set of all real continuous  functions  $Q$ on $\RR^{d}$ such that for each  fixed $x\in\RR^{d}$ and $\xi\in\EEE$,  the function $g(t):= Q(x+t\xi)$ is an algebraic polynomial of degree $< r$ in the variable $t\in\RR$.
\end{defn}

Clearly,  $\Pi_{r-1}^d(\CE)$ is a linear subspace of $C(\Og)$. Let us now state certain basic properties of this subspace in the following proposition which will be proved in Section~\ref{sec:prliminary}.
\begin{prop}\label{pro:properties}
	(i) If  $\text{span}(\mathcal{E})\not=  \RR^{d}$, then $\dim\Pi_{r-1}^d(\CE)=\infty $. \\
	(ii) If  $\text{span}(\mathcal{E})=  \RR^{d}$ i.e. $\CE$ contains some $d$ linearly independent directions, then $\Pi^{d}_{r-1}(\mathcal{E})$ is a finite dimensional space of algebraic polynomials depending on $\CE$. Moreover, $\dim\Pi_{r-1}^d(\CE)\le r^d$ and $\Pi^{d}_{r-1}(\mathcal{E})\subset \Pi^{d}_{d(r-1)}$, i.e. the monomials in any element from $\Pi^{d}_{r-1}(\mathcal{E})$ always have total degree not exceeding $d(r-1)$.
\end{prop}

In this paper, unless otherwise stated, we will always  assume that $\spn(\CE)=\RR^d$ so that $\dim \Pi_{r-1}^d(\CE) <\infty$.

\begin{rem}\label{rem-1-3}
By the definition,  if $\Og$ is the closure of a connected open set, then a function $g\in L^p(\Og)$ belongs to the space $\Pi_{r-1}^d(\CE)$ if and only if
	$ \og_{\CE}^r(g; \Og)_p=0.$
\end{rem}

Now  given a set of directions $\CE\subset \sph$, we define the best approximation of $f\in L^p(\Og)$ by functions  from the space $\Pi_{r-1}^d(\CE)$ in the $L^p$-metric by
\[
E_{r}(f;\EEE)_{L^p(\Og)}:=\inf\Bl\{ \|f-Q\|_{L^p(\Og)}:\   \  Q\in \Pi^{d}_{r-1}(\EEE)\Br\},\  \ 0<p\leq \infty.
\]
 By Remark ~\ref{rem-1-3},  we have
 $$   \og^r_{\CE}  (f; \Og)_p\leq C_{p,r} E_{r}  (f;\EEE)_{L^p(\Og)}.   $$
  Thus, an  appropriate generalization of the Whitney inequality  for directional moduli of smoothness     takes  the form
\begin{equation}\label{eqn:new whitney type}
E_{r}  (f;\EEE)_{L^p(\Og)} \leq C \og_{\CE}^r (f, \Og)_p,
	\end{equation}
	where $C>0$ is a constant independent of $f$.

For $r\in\NN$,  $0<p\leq \infty$ and a nonempty set  $\CE\subset \sph$, we define the directional  Whitney constant by
\begin{equation}
w_r(\Og;\CE)_p:= \sup\Bl \{ E_r(f;\CE)_{L^p(\Og)}:\  \ f\in L^p(\Og),\   \ \og_{\CE}^r(f; \Og)_p\leq 1\Br\}.
\end{equation}
We  say a set $\CE\subset \sph$ of directions admits a  directional Whitney inequality on a domain $\Og\subset \RR^d$ if $w_r(\Og; \CE)_p<\infty$ for all $p>0$ and $r\in\NN$.

Clearly, $w_r(\Og)_p=w_r(\Og;\sph)_p$, and \begin{equation}\label{1-6-00}
w_r(\Og;\CE)_p=w_r(\Og;\CE\cup(-\CE))_p.
\end{equation}
 Furthermore,  it can  be easily verified   that
 for each  non-singular affine transformation  $Tx=Ax+x_0$, $ x, x_0\in\RR^d$, $A\in\RR^{d\times d}$, we have
\begin{equation}\label{1-6a}
w_r(\Og; \CE)_p=w_r (T(\Og); \CE_T)_p,\  \ \text{where $\CE_T:=\{ A^{-1}\xi/\|A^{-1} \xi\|:\  \ \xi\in\CE\}$.}
\end{equation}

In this paper, we will develop a new  method which allows us  to deduce a directional  Whitney  inequality~\eqref{eqn:new whitney type} on a    domain $\Og\subset \RR^d$   from a directional  Whitney  inequality  on  a geometrically     simpler subdomain.   

\begin{rem} Ditzian and Ivanov ~\cite{Di-Iv} studied  the equivalence of the moduli of smoothness  $\omega^r(f,t)_{L^p(\Og)}$ and   $\omega_{\CE}^r(f,t)_{L^p(\Og)}$  for $p\ge1$. They proved that  under certain conditions on  the set $\CE$ of directions and the domain $\Og$,
	\begin{equation}\label{1-10}
	C^{-1} \omega^r(f,t)_{L^p(\Og)}\leq \omega_{\CE}^r(f,t)_{L^p(\Og)}\leq C \omega^r(f,t)_{L^p(\Og)},\  \ 0<t<1,\  \ p\ge 1,
	\end{equation}
	where $C>0$ is a constant independent of $t$ and $f$. This result in particular implies $\Pi^d_{r-1}(\EEE)=\Pi^d_{r-1}$.   However,  	 such an  equivalence is not applicable to the case of $p<1$, and  often  requires a larger number  of directions.  For example, for  a cube in $\RR^d$,  the cone-type conditions in~\cite{Di-Iv} may require  $\ge c 2^d$  different directions in $\EEE$. Our results in this paper show that the directional Whitney inequality \eqref{eqn:new whitney type} normally  requires fewer directions.

\end{rem}

\subsection{Organization of the paper}

This paper is organized as follows.  Several  preliminary results are proven in Section~\ref{sec:prliminary}.  A key  lemma in this section asserts  that
if    $(K, J)$  is   a pair of measurable subsets of $\RR^d$ satisfying that
$\bigcup_{j=1}^r (J-j h) \subset  K$    for some $h\in\RR^d\setminus\{0\}$ such that   $h/\|h\|\in\CE\subset \sph$, then
\begin{align}\label{5-2-eq-18-0}
w_r(K\cup J; \CE)_p^\ta \leq  1+ 2^r w_r (K; \CE)_p^\ta,\   \  0<p\leq \infty,\  \ \ta:=\min\{p,1\}.
\end{align}
This  estimate  can be applied   iteratively    to obtain  a directional  Whitney  inequality on a  general connected    domain  from a known  Whitney inequality  on a geometrically   simpler subdomain (e.g., the rectangular box). The idea here plays a crucial role in later sections.

In Section ~\ref{sec:4}, to illustrate the idea of our method,  we  generalize   the ordinary  Whitney inequality   to  a  class   of more general, possibly non-convex compact  domains.
Our proof  is simper than that in \cite{De-Le}.

In Section ~\ref{sec:5}, we prove that a set  $\CE\subset \sph$ with $\spn(\CE)=\RR^d$  is a universal set of directions that admits the directional Whitney inequality on every convex body in $\RR^d$ if and only if $\overline{\CE\cup (-\CE)}= \sph$. This means that  one has to take the domain into consideration when choosing the directions for the Whitney inequality. We also prove in  Section ~\ref{sec:5} that  if  $\CE_0\subset \sph$ is a set of directions that admits the directional Whitney inequality on a domain $\Og\subset \RR^d$, so is any larger set $\CE$ of directions containing $\CE_0$.

In Section ~\ref{sec:6},  we prove that  every set of $d$ linearly independent directions is a universal set of directions that  admits the directional Whitney inequality  on all  connected $C^2$-domains in $\RR^d$.

Finally, in Section ~\ref{sec:7},
 we study    the smallest number $\cN_d(G)\in\NN$ for each given  convex body $G\subset \RR^d$ such that there exists a set $\CE$ of $\cN_d(G)$ directions with $\spn(\CE)=\RR^d$ for which $w_r(G; \CE)_p<\infty$ for all $p>0$ and $r\in\NN$.  For $d=2$, we prove that $\cN_2(G)=2$. For $d\ge 3$,
   we connect  the number $\cN_d(G)$ with a problem in convex geometry on the  X-raying number $X(G)$  of $G\subset \RR^d$, proving  that if a convex body $G\subset \RR^d$ is  X-rayed by a finite set  of directions $\CE\subset \sph$, then $w_r(\Og; \CE)_p <\infty$ for all $r\in\NN$ and $0<p\leq \infty$. This,  in particular,  implies that $\cN_d(\Og)\leq X(\Og)$. The connection  also   allows us to show that $\cN_d(K)=d$ for a class of almost smooth convex bodies $K\subset \RR^d$, and to obtain  certain quantitative upper estimates of the number $\cN_d(G)$ for more general convex bodies  $G\subset \RR^d$.  The problem considered in this section may have  potential applications  in approximation of large data.

\section{Preliminary results}\label{sec:prliminary}

We start with the proof of Proposition~\ref{pro:properties}.

\begin{proof}[Proof of Proposition~\ref{pro:properties}]
	(i) If $\text{span}(\mathcal{E})\neq  \RR^{d}$, then there exists   $\xi_0\in \SS^{d-1}$    such that $\xi_0\cdot \xi=0$ for all $\xi\in\mathcal{E}$. This implies that    every ridge function of the form $f(x):=g(x\cdot \xi_0)$ with  $g\in C(\RR)$ is contained  in the space  $\Pi^{d}_{r-1}(\EEE)$, and thus $\dim\Pi_{r-1}^d(\CE)=\infty $.\\
	(ii) Suppose $\text{span}(\mathcal{E})=\RR^{d}$ and   $\{\xi_1,\dots, \xi_d\}\subset\CE$ is a set of $d$ linearly independent directions. It is easy to see that the statement we need to prove is invariant under any affine change of variables (in particular, the space $\Pi^{d}_{d(r-1)}$ is), so we can assume that $\xi_j$ is the $j$-th basic unit vector. Then for any $f\in \Pi^{d}_{r-1}(\EEE)$ the directional modulus of smoothness $\omega^r_\CE(f;\Omega)_\infty$ is zero for $\Omega$ being any compact parallelepiped with sides parallel to the coordinate axes, so by~\cite[Lemma 2.1]{Di96} we obtain that $f$ belongs to the space of polynomials of degree $<r$ in each of the $d$ variables. This space has dimension $r^d$ and is clearly a subspace of $\Pi^{d}_{d(r-1)}$.
\end{proof}

We will also need the lemma we just used in a somewhat more general form.

\begin{lem}\label{lem-5-3-0}\cite[Lemma 2.1]{Di96}
Let  $S$ be  a   compact 	parallelepiped in $\RR^{d}$, and let $\mathcal{E}_S\subset \sph$  denote   the set of all edge directions of $S$. Then
$$ w_r(S; \CE_S)_p\leq C(p, d, r)<\infty,\   \  r\in\NN,\   \  0<p\leq \infty. $$
\end{lem}

\begin{proof}  Lemma~\ref{lem-5-3-0} was proved in \cite[Lemma 2.1]{Di96} in the case when $S$ is a rectangular box. The  more  general  case follows from~\eqref{1-6a}.
\end{proof}

Our second lemma will be used repeatedly in this paper and is the heart of the matter.

\begin{lem} \label{lem-3-3} Let $r\in\NN$ and  $\CE\subset \sph$. Assume that   $(K, J)$  is   a pair of measurable subsets of $\RR^d$ satisfying that  $J\subset  \bigcap_{j=1}^r (K+jh)$  for some  $h\in\RR^d\setminus\{0\}$ such that $h/\|h\|\in\CE$. Then  for any $0<p\leq \infty$,
	\begin{align}\label{5-2-eq-18}
	w_r(K\cup J; \CE)_p^\ta \leq  1+ 2^r w_r (K; \CE)_p^\ta,\   \   \   \  \ta:=\min\{p,1\}.
	\end{align}
\end{lem}
\begin{proof}     Given a function   $F:K\cup J \to \RR$, we have
	\begin{align*}
		(-1)^r\tr_h^r F (\xi-rh)=\tr_{-h}^r  F(\xi) =\sum_{j=0}^r (-1)^j \binom{r} j  F(\xi-jh),\   \ \xi\in J,
	\end{align*}
and thus,
	$$|F(\xi)| \leq |\tr_{h}^r F(\xi-rh)|+  \sum_{j=1}^r\binom{r}j |F(\xi-jh)|,\   \  \forall \xi\in J.$$
Since  $J\subset \bigcap_{j=1}^r (K+jh)$, it follows that
	$$ \|F\|_{L^p(J)}^\ta \leq  \|\tr_{h}^r F\|^\ta_{L^p (J-rh)}  +  (2^r-1)\|F\|^\ta_{L^p(K)},$$
which in turn implies
	\begin{equation}\label{2-3-18-55}
		\|F\|_{L^p(K\cup J)}^\ta \leq  \|\tr_{h}^r F\|^\ta_{L^p (J-rh)}  + 2^r \|F\|_{L^p(K)}^\ta.
	\end{equation}
	Since $\tr_h^r Q=0$ for each $Q\in\Pi_{r-1}^d(\CE)$, using  \eqref{2-3-18-55} with $f-Q$ in place of $F$,  we deduce 	
	\begin{align*}\label{5-2-eq-0-18}
	\|f-Q\|^\ta_{L^p(K\cup J)} &\leq \| \tr_{h}^r f\|^\ta_{L^p({J-rh})}  + 2^r \|f-Q\|^\ta_{L^p(K)},\   \ \forall Q\in\Pi_{r-1}^d(\CE).
	\end{align*}
	It follows that
	 	\begin{align*}
	 	 E_{r} (f;\EEE)_{L^p(K\cup J)}^\ta &\leq   \og_{h/\|h\|}^r (f; K\cup J)^\ta_p+ 2^r E_r(f,\CE)_{L^p(K)}^\ta\\
	 	 & \leq   \og_{h/\|h\|}^r (f; K\cup J)^\ta_p+ 2^rw_r(K;\CE)_p^\ta \og_\CE^r(f; K)_p^\ta\\
	 	 &\leq \og_{\CE}^r (f; K\cup J)^\ta_p\Bl( 1+ 2^rw_r(K;\CE)_p^\ta\Br),
	 	\end{align*}
	 	which  proves \eqref{5-2-eq-18}.
\end{proof}

To   establish a directional  Whitney type inequality   on a   domain $\Og\subset \RR^{d}$ for a set of directions $\CE\subset \sph$,  we   often need to  apply the estimate \eqref{5-2-eq-18}  iteratively.
Indeed, we have  the following useful lemma:

\begin{lem}\label{rem:whitney key tool seq appl}
Let $\CE\subset \sph$ be a set of directions. Assume that  	$\Og\subset \RR^d$ is   a finite union  $\Og=\bigcup_{k=0}^m E_k$ of measurable subsets $E_k\subset \RR^d$ satisfying  that
for each $1\leq k\leq m$,
there exists a vector $h_k\in \RR^d\setminus\{0\}$ whose direction lies in $\CE$ such that
\begin{equation}\label{7-2-0-18}
E_{k} \subset  \bigcap_{j=1}^r (\Og_{k-1}+jh_k).\   \
\end{equation}
where
$\Og_k:=\bigcup_{j=0}^k E_j$, $k=0,1,\dots, m$.
Then
\begin{equation}\label{3-5-0}
w_r(\Og; \CE)_p^\ta \leq 2^{mr} w_r(\Og_0; \CE)_p^\ta +\f { 2^{mr} -1}{2^r-1}.
\end{equation}

\end{lem}

\begin{proof}
	With the decomposition $\Og=\bigcup_{k=0}^m E_k$  , we may apply \eqref{5-2-eq-18}  recursively to the pairs of sets $(K, J)=(\Og_{k-1}, E_k)$, $k=1,2,\dots, m$ to obtain
$$w_r(\Og_k; \CE)_p^\ta \leq 1+ 2^r w_r(\Og_{k-1};\CE)_p^\ta,\   \ k=1,2,\dots, m.$$
 \eqref{3-5-0} then follows.
\end{proof}	

\begin{rem} To apply  Lemma ~\ref{rem:whitney key tool seq appl}  to establish a directional Whitney inequality in a domain,  the crucial   step is  to decompose the domain $\Og$ as a finite union $\Og=\bigcup_{k=0}^m E_k$ so that  the condition \eqref{7-2-0-18} is satisfied.
	With this approach, we often choose   an   initial set  $E_0$ to be geometrically simpler so that its  directional  Whitney constant $w_r(E_0;\CE)_p$ is known to be finite.   For instance, according to  Lemma~\ref{lem-5-3-0},  we may choose $E_0$ to be  any  compact
parallelepiped in $\RR^{d}$.
We need to assume that the  direction of $h_j$ for each $j$
 lies  in the set $\EEE$.    If  too many directions are involved, we may end up with $\Pi^{d}_{r-1}(\EEE)=\Pi^{d}_{r-1}$.

\end{rem}

For later applications, we introduce the following definition.

\begin{defn}\label{def-3-4} Given $\xi\in\sph$, we say $G \subset \RR^{d}$  is  a regular   $\xi$-directional domain with parameter $L\ge 1$      if   there exists a  rotation $\pmb{\rho}\in SO(d)$ such that\begin{enumerate}[\rm (i)]
		\item
	 $\pmb{\rho}(0,\dots, 0,1)=\xi$, and  $ G$ takes the form
	 \begin{equation}\label{3-6}
	 G:=\pmb{\rho}\Bl(\{(x, y):  \  x\in D,\   g_1(x)\leq y\leq g_2(x)\}\Br),
	\end{equation}
where  	$D\subset \RR^{d-1}$ is compact    and  $g_i: D\to\RR$ are measurable;
\item
  there exist an affine function $H:\RR^{d-1}\to\RR$ and a constant $\da>0$  such that   $S\subset G\subset S_L$, where
\begin{align}
\pmb{\rho}^{-1} (S):&=\{(x,y):\  \  x\in D,\  \  H(x)-\da\leq y\leq H(x)+\da\},\label{3-7}\\
\pmb{\rho}^{-1}  (S_L):&= \{(x,y):\  \  x\in D,\  \  H(x)-L\da\leq y\leq H(x)+L\da\}.\label{3-8}
\end{align}  	
	\end{enumerate}	
In this case, we say $S$ is the base of $G$.
\end{defn}

\begin{lem}\label{cor-7-3}
	Let $G\subset \RR^d$ be a regular  $\xi$-directional  domain with parameter $L\ge 1$ and base $S$   as given in Definition~\ref{def-3-4} for some $\xi\in\sph$.  Let  $\CE\subset\sph$ be a set of directions containing $\xi$.  	Assume that   $K$  is  a measurable subset of $\RR^d$ such that $S\subset  K\cap G$   and $w_r(K; \CE)_p<\infty$ for some $r\in\NN$,  $0<p\leq \infty$.
	Then
\begin{equation}\label{3-9-a}
w_r(G\cup K; \EEE)_p\leq C_{p,r} L^{r-1+2/p}(1+ w_r(K; \CE)_p),
\end{equation}
where the constant $C_{p,r}$ depends only on $p$ and $r$.

\end{lem}

\begin{proof} Without loss of generality, we may assume $\xi=e_d:=(0,\dots, 0,1)$ and $\pmb \rho=I$ in Definition ~\ref{def-3-4}. 	
	Let $A_1:=w_r(K;\CE)_p$, let $f\in L^p(K\cup G)$ and let $Q\in \Pi_{r-1}^d(\CE)$ be such that
	\begin{equation}\label{3-10}
	\|f-Q\|_{L^p(K)} \leq A_1 \og_\CE^r(f; K)_p.
	\end{equation}
	It is enough to show that
	\begin{equation}\label{3-11}
	\|f-Q\|_{L^p(G)} \leq C_{p,r} L^{r-1+2/p} (1+A_1)  \og_\CE^r(f; K\cup G)_p.
	\end{equation}

	For simplicity, we may assume that $p<\infty$.  (The case $p=\infty$  can be treated similarly.)
	By the Whitney inequality in one variable (see \cite[p. 374, p. 183]{De-Lo}), for each fixed $x\in D$, there exists $\vi_x\in \Pi_{r-1}^1$ such that \begin{align*}
	\int_{g_1(x)}^{g_2(x)} &|f(x,y) -\vi_x(y)|^p \, dy \leq 	C_{p,r}  \sup_{0<hr<g_2(x)-g_1(x)}  \int_{g_1(x)}^{g_2(x)-rh} |\tr_{he_d}^r f(x,y)|^p dy \\
	&\leq\f { C_{p, r}} {g_2(x)-g_1(x)}  \int_0^{(g_2(x)-g_1(x))/r}\Bl[ \int_{I_{x, hr}} |\tr_{ h e_d}^r f(x,y)|^p dy\Br] dh,
	\end{align*}
	where
$I_{x,rh}:=\bigl\{ y\in \RR:\  \ y, y+ rh \in [g_1(x), g_2(x)]\bigr\}$,and the second  step uses (5.17) of \cite[p. 373]{De-Lo}.
Since $S\subset G\subset S_L$, we have $\da\leq g_2(x)-g_1(x)\leq 2 L\da$ for any $x\in D$. It follows that
	\begin{align}
\int_D	\int_{g_1(x)}^{g_2(x)} |f(x,y) -\vi_x(y)|^p \, dy dx
& \leq \f{C_{p,r}}\da  \int_0^{2L\da } \int_D \int_{I_{x,rh}} |\tr_{h e_d}^r f(x,y)|^p dy dx dh\notag\\
 & \leq C_{p,r} L\cdot  \og_{e_d} ^r (f; G)_p^p.\label{3-11}
	\end{align}
	Thus,
		\begin{align*}
	\|f-Q\|_{L^p(G)} ^p &=\int_D \int_{g_1(x)}^{g_2(x)} |f(x,y)-Q(x,y)|^p \, dy dx\\
	& \leq C_{p,r} L\cdot \og_{e_d}^r(f; G)_p^p + 2^p \int_D \int_{g_1(x)}^{g_2(x)} |\vi_x(y)-Q(x,y)|^p \, dydx.
	\end{align*}
	Since $e_d\in\CE$, $Q(x,y)\in \Pi_{r-1}^d(\CE)$ is an algebraic polynomial of degree $<r$ in the  variable $y\in\RR$.  It follows  by the Remez inequality in one variable that
		\begin{align*}
	 \int_D& \int_{g_1(x)}^{g_2(x)} |\vi_x(y)-Q(x,y)|^p \, dydx\leq C_{p,r} L^{(r-1)p+1}  \int_D \int_{H(x)-\da}^{H(x)+\da} |\vi_x(y)-Q(x,y)|^p \, dydx,
	  \end{align*}
	 which,  using \eqref{3-11} and \eqref{3-10}, is estimated above by
	 \begin{align*}
	 		&\leq C_{p,r} L^{(r-1)p+2}\cdot \og_{e_d}^r(f; G)_p^p+ C_{p,r} L^{(r-1)p+1} \|f-Q\|_{L^p(S)}^p\\
		& \leq C_{p,r} L^{(r-1)p+2}\cdot \og_{e_d}^r(f; G)_p^p+ C_{p,r} L^{(r-1)p+1} A_1^p \og_{\CE} ^r (f; K)_p^p\\
		&		 \leq C_{p,r} L^{(r-1)p+2} (1+A_1^p) \og_{\CE}(f; K\cup G)^p_p.
		\end{align*}
This proves \eqref{3-11}.
\end{proof}

%
%
%

\begin{rem}\label{rem-7-4}
	
	If, in addition, we assume the domain  $D$ in Definition ~\ref{def-3-4}  is a  compact 	parallelepiped  in $\RR^{d-1}$,  then we may choose $K=S$ and use   Lemma ~\ref{lem-5-3-0}  to obtain
	\begin{equation}
	w_r(G; \EEE)_p\leq C(p,r) L^{r-1+2/p}<\infty,
	\end{equation}
	where $\CE\subset\sph$ is the  set of the edge directions of $S$.
	
%
%
\end{rem}
%
%

\section{ The ordinary  multivariate  Whitney inequality }\label{sec:4}

It  was shown in \cite{De-Le} that for all $r\in\NN$ and $p>0$,   $\sup_{K} w_r(K)_p\leq C_{p,d}<\infty$, where the supremum is taken over all convex bodies in $\RR^d$.
To illustrate the main idea of our method,  we extend this result to a class of more general, and possibly non-convex  domains.

\begin{defn}
	A domain $\Og\subset\RR^d$ is star-shaped with respect to a ball $B$ in $\RR^d$ if for each $x\in\Og$, the closed convex hull of $\{x\}\cup B$ is contained in $\Og$. Given a constant $R>1$, we say a domain $\Og\subset\RR^d$ belongs to the  class $\CS^d(R)$ if there exists an affine transformation $T:\RR^d\to \RR^d$ such that $T(\Og)$ is a star-shaped domain with respect to the unit ball $B_1[0]$ in $\RR^d$ and is contained in the ball $B_R[0]$.
\end{defn}

\begin{rem}
	 There are many non-convex sets in the  class $\CS^d(R)$. A simple example is as follows. Let $A=\{x_1,x_2\}\subset R\SS^{d-1}$ be a set of two points distance $\delta$ apart, where $0<\delta<2\sqrt{R^2-1}$. Take $\Omega$ to be the union over $x\in A$ of the convex hulls of $\{x\}\cup B_1[0]$. Then $\Omega$ is not convex and belongs to $\CS^d(R)$. Various other choices of $A$ are possible.
\end{rem}

According to  John's theorem (\cite{Jo}, \cite[Proposition 2.5]{De-Le}), every convex body  in $\RR^d$ belongs to the class $\CS^d(d)$.
\begin{thm}\label{De-Le-1}  Given $r\in\NN$,  $0<p\leq \infty$ and $R>1$, there exists a constant $C(p, d,r,R)$ depending only on $p, d, r$ and $R$ when $R\to \infty$  such that
	\begin{equation}
\sup_{\Og\in\CS^d(R)}	w_r(\Og)_p\leq C(p, d, r, R)<\infty.
	\end{equation}
\end{thm}

\begin{proof}By \eqref{1-6a}, without loss of generality, we may assume that $\Og$ is a star-shaped domain  with respect to the unit  ball $B_1[0]$ and $\Og\subset B_R[0]$. 	Following  \cite{Di-Pr08}, we may   decompose $\Og$ as a finite union $\Og=\bigcup_{j=0}^m G_j$ of $(m+1)\leq (C Rd)^{d-1}$ regular directional domains $G_j$ with parameters $\leq C_d R$ such that $G_0:=[-\f 1{\sqrt{d}}, \f 1{\sqrt{d}}]^{d}$, and   the  base of  each   $G_j$, $j\ge 1$ is contained in $G_0$. Indeed, the decomposition  can be constructed as follows. First,   cover the sphere $\{ \xi\in\RR^{d}:\  \ \|\xi\|=R\}$ with
	$m\leq (CRd)^{d-1}$ open balls  $B_1, \dots, B_m$  of radius $1/ (2d)$ in $\RR^{d}$.    Denote by $\xi_{j}$ the unit vector pointing from the origin to the center of the ball $B_j$, and let $I_j$ denote a $(d-1)$-dimensional cube  with  side length $1/d$ centered at the origin and perpendicular to $\xi_j$. For each point $x\in\Omega\setminus B_{1/(2d)}[0]$ the perpendicular from $x$ onto $I_j$ belongs to the closed convex hull of $x$ and $B_{1/(2d)}[0]$, where $j$ is such that $B_j$ contains $x\frac{R}{\|x\|}$. Indeed, this readily follows from the fact that the two parallel lines having the direction $\xi_j$ and passing through the origin and through $x$ are at most $\frac1{2d}$ far apart. Define $ G_0:=[-\f 1{\sqrt{d}}, \f 1{\sqrt{d}}]^{d}$ and
	$$ G_j:=\Bl\{ \eta+t\xi_j\in \Og:\  \  \eta\in I_j,\   \  t\ge 0\Br\},\  \ j=1,\dots, m.$$
	Since $\Og$ is star-shaped with respect to the unit ball $B_1[0]$ and  $B_1[0]\subset \Og\subset B_{R}[0]$, it is easily seen  that  $\Og=\bigcup_{j=0}^m G_j$ and each   $G_j$, $j\ge 1$  is a regular $\xi_j$-directional  domain with parameter $\leq C_d R$ and  base
	$$S_j:=\{\eta+t\xi_j:\  \ \eta\in I_j, \  \  0\leq t\leq \sqrt{d-1}/(2d)\}\subset G_0.$$

	Now let $\Og_j:=\bigcup_{k=0}^j G_k$ for $j=0,1,\dots, m$. Applying  Lemma~\ref{cor-7-3} iteratively to the pairs $(G, K)=(G_j, \Og_{j-1})$ for $j=1,\dots, m$, we obtain
	$$w_r(\Og_j)_p=w_r(\Og_{j-1}\cup G_j)_p \leq C(p,r,d,R) \bigl( 1+w_r(\Og_{j-1})_p\bigr),\  \ j=1,2,\dots, m.$$
	By Lemma ~\ref{lem-5-3-0}, this implies that
	$$ w_r(\Og)_p=w_r(\Og_m)_p \leq  C(p,r,d,R) \bigl( 1+w_r(\Og_{0})_p\bigr)\leq C(p,r,d,R)<\infty. $$

\end{proof}

\begin{rem}\label{rem-4-3}
	The above proof also yields a directional  Whitney inequality with   the number of directions required in the directional modulus   $\ge (cdR)^{d-1}$. Indeed, from the above proof,  if $\Og\subset B_R[0]$ is star-shaped with respect to the unit ball $B_1[0]$, and $\CE\subset \sph$ is a finite set of directions satisfying that $\sph\subset \bigcup_{\xi\in \CE} B_{1/{2Rd}}(\xi_j)$, then $w_r(\Og;\CE)_p\leq C(p, r, d, R)<\infty$ for all $r\in\NN$ and $0<p\leq \infty$.
	
\end{rem}

\section{On choices of directions}\label{sec:5}

As pointed out in Remark ~\ref{rem-4-3},  the proof of  Theorem ~\ref{De-Le-1}  yields a directional  Whitney inequality on each convex body, but with the set of directions  depending on the domain.
It is therefore  natural to ask whether there exists a  universal    set of directions $\CE\neq \sph $  for which $\spn(\CE)=\RR^d$ and $w_r(\Og;\CE)_p<\infty$ holds for all $r\in\NN$, $p>0$ and every convex body $\Og\subset \RR^d$.
We will   prove two results related to this question in this section.
\begin{thm}\label{thm-4-1}
	 Let $d\ge 2$ and let $\CE\subset \sph$  be a   set of directions such that $\spn(\CE)=\RR^d$ and   $\overline{\CE\cup (-\CE)}\neq \sph$. Then there exists a convex body $K\subset \RR^d$ such that
	 $w_r(K; \CE)_\infty=\infty$ for all $r\in\NN$. 	
\end{thm}

By Theorem ~\ref{thm-4-1}, given any $r\in\NN$, there does not exist a  universal finite    set of directions $\CE $  such that $\spn(\CE)=\RR^d$ and  $w_r(\Og;\CE)_\infty <\infty$  for every convex body $\Og\subset \RR^d$. Thus, one has   to take the domain into consideration  when choosing the directions for the Whitney inequality.

Using Theorem ~\ref{thm-4-1},  Remark ~\ref{rem-4-3} and  John's theorem (\cite[Proposition 2.5]{De-Le}) for convex bodies, we immediately derive   the following characterization of   universal   sets  of directions $\CE$ for which $w_r(\Og;\CE)_p<\infty$  for all $r\in\NN$, $p>0$ and every convex body $\Og\subset \RR^d$.

\begin{cor}
 Let $d\ge 2$ and let $\CE\subset \sph$  be a   set of directions such that $\spn(\CE)=\RR^d$.  Then in order that $w_r(\Og; \CE)_p<\infty$ for every $r\in \NN$, $p>0$ and every convex body $\Og\subset \RR^d$, it is necessary and sufficient that $\overline{\CE\cup (-\CE)}=\sph$.	
\end{cor}

Our second result shows that  the directional Whitney inequality  remains valid if the set of directions is enlarged, which, in particular,  implies  that if $\CE_0\subset \sph$ is a universal set of directions for  the directional Whitney inequality on a class of domains, so is any larger set of directions $\CE\supset \CE_0$.
\begin{thm}\label{thm:enlarged set of directions}
	Let  $G$ be the closure of a bounded connected nonempty open set in $\RR^d$ .   Assume that  $\EEE_0\subset\SS^{d-1}$  is a set of directions  such that $\spn(\CE)=\RR^d$  and   $w_r(G; \CE_0)_p<\infty$ for some $r\in\NN$ and $0<p\leq \infty$. Then for any   $\CE_0\subset \EEE\subset \SS^{d-1}$,
	\begin{equation}\label{eqn:whitney for all}
	w_r(G; \CE)_p\leq C (w_r(G; \CE_0)_p+1),
	\end{equation}	
	where the constant $C>0$ depends only on $p, r$ and the set $G$.
\end{thm}

\subsection{ Proof of Theorem ~\ref{thm-4-1}}

\begin{proof}

By \eqref{1-6-00}, we may assume, without loss of generality that $\CE=-\CE$ and   $\overline{\CE}\neq \sph$.  Our aim is to find a  convex body  $K\subset \RR^d$ such that
\begin{equation}\label{4-4}
 w_r(K; \CE)_\infty=\infty.
\end{equation}

	Fix a direction  $\xi\in \sph\setminus \overline{\CE}$.  Clearly,  there exists a constant $\da\in (0,1)$ such that
	\begin{equation}\label{4-5}
	\eta\cdot \xi\leq 1-\da,\   \forall \eta\in\CE.
	\end{equation}
	Let $0<\va<\da$  be a constant, and define
	$$K:=\Bl\{x\in\RR^d:\  \ \|x\|(1-\va)\le x \cdot \xi \le 1\Br\}.$$
	Then $K$ is a convex body in $\RR^d$.
	Consider a sequence of continuous functions on $K$ given by
	$$f_n(x): =g_n(x\cdot \xi),\  \ x\in K,\   \ n=1,2,\dots,$$
	where $g_n(0)=-n$ and
	$g_n(t):=
	\max \{ -n, \ln t\}$ for $t>0$.
	Clearly,  to show \eqref{4-4}, it is sufficient to prove that
	\begin{equation}\label{4-6}
	\sup_n \omega_{\CE}^r(f_n;K)_{\infty}\leq 	C_r<\infty,
	\end{equation}
	and
	\begin{equation}\label{5-1}
	\lim_{n\to\infty} E_{r}(f_n;\EEE)_{L^\infty(K)}=\infty.
	\end{equation}

	To this end,   we first  claim that if $x,y\in K$ and $(y-x)/\|y-x\|\in\CE$, then
	\begin{equation}\label{4-8}
	c^{-1} ( x\cdot \xi)\leq y\cdot \xi \leq c (x\cdot \xi),
	\end{equation}
	where
	$c= \f {2-\da}{\da-\va} \f 1 {1-\va}$.
	Indeed,
	since
	$\f {y-x}{\|y-x\|} \in\CE,$
	we obtain from \eqref{4-5} that
	$$ (y-x)\cdot \xi =\|y-x\|\f {y-x}{\|y-x\|} \cdot \xi\leq (1-\da) \|y-x\|\leq (1-\da) (\|x\|+\|y\|).$$
	On the other hand, by the definition of the set $K$,
	$$(y-x)\cdot \xi=y\cdot \xi-x\cdot \xi\ge (1-\va)\|y\|-\|x\|.$$
	Thus,
	$$(1-\va) \|y\|-\|x\|\leq (1-\da)(\|x\|+\|y\|),$$
	which  implies
	\begin{equation*}
	\|y\|\leq \f {2-\da}{\da-\va} \|x\|.
	\end{equation*}
	\eqref{4-8} then follows by symmetry and the inequality
	$$(1-\va)\|z\|\leq z\cdot\xi \leq \|z\|,\  \ \forall z\in K.$$

	Next, we prove \eqref{4-6}.
	Since
	\[ \omega_{\CE}^r(f_n;K)_{\infty}  \leq  C_r \omega_{\CE}(f_n;K)_\infty,\]
	it is enough to prove \eqref{4-6} for $r=1$.
	Assume that $x, x+s\eta\in K$ for some $s>0$ and $\eta\in\CE$.
	If both $x\cdot \xi$ and $ (x+s\eta)\cdot \xi$ lie in the interval $[0, e^{-n}]$, then
	$\tr_{s\eta} f(x) =0$. If both  $x\cdot \xi$ and $(x+s\eta)\cdot \xi$ lie in the interval $[e^{-n},1]$,  then using  \eqref{4-8} and  the mean value theorem, we have
	\begin{align*}
	|\tr_{s\eta} f(x)|&=\Bl|g_n (x\cdot \xi) -g_n \bigl( (x+s\eta) \cdot \xi\bigr)\Br|\leq C \f { (x+s\eta)\cdot \xi+x\cdot \xi}{x\cdot \xi} \leq C<\infty.
	\end{align*}
	If only one of the numbers $x\cdot \xi$ and $ (x+s\eta)\cdot \xi$ lies in the interval $[e^{-n},1]$, say,
		 $(x+s\eta)\cdot \xi\ge  e^{-n}$ and $x\cdot \xi <e^{-n}$, then there exists a number $0\leq s_0\leq s$ such that  $\xi\cdot (x+s_0 \eta) =e^{-n}$, and hence,
	\begin{align*}
	|\tr_{s\eta} f(x)|&=|f(x+s \eta)-f(x+s_0\eta)|=|\tr_{(s-s_0) \eta} f (x+s_0\eta) |\leq C.
	\end{align*}
	Thus, in all the cases, we have shown that
	$$ \sup_{\eta\in \CE} \max_{x, x+s\eta \in K}  |f(x+s\eta)-f(x)|\leq C.$$
 \eqref{4-6} for all $r\in\NN$ then follows.

	Finally, we prove \eqref{5-1}.
	Let  $Q\in \Pi_{r-1}^d(\CE)$ be such that $E_r(f_n)_{L^\infty(K)} =\|f_n-Q\|_{L^\infty(K)}$.
	Let $h=\f 1 {dr} \xi$. Since $\spn (\CE)=\RR^d$, we have $\Pi_{r-1}^d(\CE) \subset \Pi_{d(r-1)}^d$.  Thus,  for $n>\ln (dr)$,
	\begin{align*}
	C_r \|f_n-Q\|_{L^\infty(K)}&\ge |	\tr_{h}^{dr} (f_n-Q)(0)| =|\tr_h^{dr} f_n(0)|\\
	&=\Bl|\sum_{j=0}^{dr} (-1)^{j+1} \binom{dr}j f_n ( j h)\Br| \ge n-2^{dr} \ln(dr)\\
	&\to\infty,\   \   \ \text{as $n\to\infty$}.
	\end{align*}
	This shows \eqref{5-1}.
	\end{proof}

\subsection{Proof of Theorem ~\ref{thm:enlarged set of directions}}
\begin{proof}
	Let $A_0:=w_r(G; \CE_0)_p$.  Since  $\spn (\CE_0)=\RR^d$, it follows  by Proposition~\ref{pro:properties} that
	$\Pi_{r-1}^d(\CE_0)$   is a finite-dimensional vector space, and $\Pi_{r-1}^d(\CE)\subset \Pi_{r-1}^d(\CE_0)$.
	We claim that
	\begin{equation}\label{eqn:whitney for subspace}
	E_{r}  (f;\EEE)_{L^p(G)} \leq C \og_{\CE}^r (f; G)_p \qtq{for any} f\in \Pi^d_{r-1}(\EEE_0),
	\end{equation}	
	where $C>0$ depends only on $p$, $r$ and $G$.
	Indeed, from  the definition of $\Pi^d_{r-1}(\EEE)$ and Remark ~\ref{rem-1-3}, for any $g\in C(G)$,
	\begin{equation*}
	\og^r_{\CE} (g; G)_p=0 \iff g\in\Pi^d_{r-1}(\EEE) \iff E_{r} (g;\EEE)_{L^p(G)}=0.
	\end{equation*}
	This implies  that both the  mappings
	\begin{equation*}
	g \mapsto \og_{\CE}^r (g; G)_p^{\min\{1,p\}} \qtq{and}
	g \mapsto E_{r} (g;\EEE)_{L^p(G)}^{\min\{1,p\}}
	\end{equation*}
	are quasi-norms on the quotient space $\Pi^d_{r-1}(\EEE_0)/\Pi^d_{r-1}(\EEE)$. Since this space is finite-dimensional, the norms are equivalent and~\eqref{eqn:whitney for subspace} follows.	
	
	Now we can show~\eqref{eqn:whitney for all}. Fix $f\in L^p(G)$. Let $Q \in \Pi^d_{r-1}(\EEE_0)$ be such that
	$$
	\|f-Q\|_{L^p(G)}\le A_0 \omega_{\CE_0}^r(f;G)_p.
	$$
	Next, 	by~\eqref{eqn:whitney for subspace}, there exists  $R\in \Pi^d_{r-1}(\EEE)$  such that
	$$
	E_r(Q; \CE)_p=	\|Q-R\|_{L^p(G)}\le C \omega_{\CE}^r(Q;G)_p.
	$$
	Thus,
	\begin{align*}
	E_{r}(f;\EEE)_{L^p(G)} & \le \|f-R\|_{L^p(G)} \le \|f-Q\|_{L^p(G)}+\|Q-R\|_{L^p(G)} \\
	&\le A_0 \omega_{\CE_0}^r(f;G)_p + C \omega_{\CE}^r(Q;G)_p \\
	&\le A_0\omega_{\CE_0}^r(f;G)_p + C \omega_{\CE}^r(f;G)_p + C\|f-Q\|_{L^p(G)}\\
	&\le C A_0 \omega_{\CE_0}^r(f;G)_p + C \omega_{\CE}^r(f;G)_p \le C (A_0+1)\omega_{\CE}^r(f;G)_p.
	\end{align*}
\end{proof}

\section{ Universal sets of directions for  smooth domains} \label{sec:6}
By Theorem ~\ref{thm-4-1}, one cannot  find a  finite set of directions $\CE$  satisfying  $\spn(\CE)=\RR^d$ that   admits the directional Whitney inequality on  every convex body in $\RR^d$. It is therefore natural to ask that for which domains $G\subset \RR^d$  a  given  set of $d$ linearly independent directions  $\CE\subset\sph$  admits the directional   Whitney inequality.
In this section, we will give an affirmative answer to this question,
proving  that if Lip-2 condition  is imposed  on the domain, then the choice of the directions can be an  arbitrary set of $d$ linearly independent directions.

\begin{defn}\label{def-6-1}
	A compact set $G\subset\RR^d$ is said to be a Lip-2 domain with parameter $L>1$ if there exists a constant $\da>0$ such that $\diam(G)\leq L\da$, and
	each point  $x\in G$ is contained in a  closed   ball $B_x\subset G$ of radius $\da$.
\end{defn}

As is well known, every connected  $C^2$-domain is a Lip-2 domain.

%
%




\begin{thm}\label{cor-7-8-18}
		Let $\mathcal{E}\subset \SS^{d-1}$ be a set of $d$ linearly independent directions with  	
		\begin{equation}\label{7-8-0-18}
		\min_{x\in\SS^{d-1}} \max_{\xi\in \mathcal{E}} |\xi\cdot x|\ge \va_0>0.	
		\end{equation}
	If  $G\subset \RR^d$ is  a connected and  compact Lip-2 domain with parameter $L>1$, then  for any $0<p\leq \infty$, $r\in\NN$ and $f\in L^p(G)$,
	\begin{equation}\label{7-8-18}w_r(G;\CE)_p \leq C(p,d,r,L, \va_0) <\infty.
	\end{equation}
\end{thm}

		Note that according to  Theorem ~\ref{thm-4-1},
		Theorem ~\ref{cor-7-8-18} is not true without the Lip-2 assumption, at least for $p=\infty$.

%
%
%

The proof of Theorem ~\ref{cor-7-8-18} relies on two technical  lemmas, which will be stated in the following subsection.

\subsection{Two technical lemmas}

Given $a>0$ and  a ball $B=B_\da(x_0)$,  we denote by $a B$ the  dilation $B_{a\da} (x_0)$.

\begin{lem}\label{lem-6-3} Let $G$ be a Lip-2 domain in $\RR^d$ with parameter $L\ge 1$ and constant $\da>0$ as given in Definition ~\ref{def-6-1}.  Let $\va\in (0,1)$ be an arbitrarily given constant. Then  the set  $G$ can be represented   as a finite union $G=\bigcup_{j=1}^m E_j$ of possibly repeated subsets $E_j\subset G$ such that
\begin{enumerate}[\rm (i)]
	\item $m\leq C(d,L/\va)<\infty$;
	\item for each $1\leq j\leq m$,   there exists  an open ball $B_j$ of radius $\da$ such that
	$B_j \subset E_j\subset (1+\va) B_j$;
		\item  $E_j\cap E_{j+1}$ contains an open ball of radius $ \va \da/4$ for each $1\leq j<m$.
\end{enumerate}

\end{lem}

\begin{proof}
	First, we  cover the domain  $G$ with $m_0\leq C_d (L/\va)^d$ open balls $B_{\va\da/4} (y_j)$, $j=1,\dots, m_0$ with   centers   $y_1,\dots, y_{m_0}\in G$.   For each $j$, we can find   an open ball $B_j=B_\da(x_j)$  of radius $\da$ such that $x_j\in  \overline{B_j}\subset G$,
	implying that $B_{\va\da /4} (y_j) \subset (1+\f \va4) B_j$. Thus,
	$$G= \bigcup_{j=1}^{m_0}  G_j,\   \ \text{where}\  \  G_j :=( (1+\f \va 4) B_j)\cap G.$$
	Since $G$ is connected and each $G_j$ is open relative to the topology of $G$,
	we  can form a sequence of sets $\{\wh E_j\}_{j=1}^{m}$ from possibly repeated copies of  the sets $G_j$, $1\leq j\leq m_0$   such that  	$m\leq 2m_0^2$,
	$$
	G=\bigcup_{j=1}^m \wh E_j \   \  \text{and}\   \   \wh E_j\cap \wh E_{j+1} \neq \emptyset \   \ \  j=1,2,\dots, m-1.$$
	
	Next, for each $1\leq j\leq m$, we define
	$$G_j^\ast := \bigl( (1+\va) B_j\bigr) \cap G.$$
	Then
	$$ B_j \subset G_j^\ast \subset (1+\va)B_j.$$
	If $G_i\cap G_j\neq \emptyset$, then
	$$\bigl((1+\f \va 4) B_i\bigr)\cap \bigl((1+\f \va 4) B_j\bigr)\neq \emptyset,$$
	and hence
	$\|x_i-x_j\|\leq (2+\f \va 2) \da$. Since
	$$ G_i^\ast\cap G_j^\ast =(1+\va) B_i\cap G_j^\ast \supset (1+\va)B_i \cap B_j,$$
	this implies that
	$G_i^\ast\cap G_j^\ast$ contains a ball of radius $\va \da/4$ if $G_i\cap G_j\neq \emptyset$.

	Finally, to complete the proof,  we define $E_j=G_i^\ast$ if $\wh {E}_j =G_i$.  	
\end{proof}

 \begin{lem}\label{lem-3-2-0} Let   $\mathcal{E}\subset \SS^{d-1}$ be a set  of $d$ linearly independent directions such that  $\min_{x\in\SS^{d-1}} \max_{\xi\in \mathcal{E}} |x\cdot \xi|\ge \va_0>0$.
 	 Let $r\in\NN$ and  $\sa:=1+\f {\va_0^2}{4r}$. Let $(S_0, E)$ be a pair of bounded measurable subsets of $\RR^d$. Let $L\ge 1$ be a given parameter.  Assume that there exists an open ball $B$ of radius $L\da$ for some constant $\da>0$ such that $B \subset E\subset \sa B$ and $ S_0\cap B$ contains an open ball $B_0$ of radius $\da$.
 	Then for any $0<p\leq \infty$,
 	\begin{equation}\label{6-5-0}
 	w_r(S_0\cup E;\CE)_p \leq C(L,d,p,r,\va_0) \Bl( 1+w_r(S_0; \CE)_p\Br).
 	\end{equation}
 \end{lem}

 The proof of Lemma~\ref{lem-3-2-0} is quite technical, so we postpone it   until Subsection~\ref{subsec-6-3}. For the moment, we take it for granted and proceed with the proof of Theorem ~\ref{cor-7-8-18} in the next subsection.
 \subsection{Proof of Theorem ~\ref{cor-7-8-18}}

 \begin{proof}

 	First, we prove that  $w_r(B; \CE)_p\leq C(p, d,r,\va_0)<\infty$  for every open ball $B\subset \RR^d$.
 	Since  dilations   and translations  do not change the directions in the set $\mathcal{E}$,
 	we  may  assume  that $B=B_1(0)$.
 	
 	Let $\CE=\{\xi_1,\dots, \xi_d\}$. Let $A$ be the $d\times d$ matrix whose $j$-th column vector is $\xi_j$. Then by~\eqref{7-8-0-18}
 	\[
 	\|A^tx\|=\Bigl( \sum_{j=1}^d |x\cdot\xi_j|^2\Bigr)^{\f12} \ge \eps_0\|x\|, \quad \forall x\in\R^{d},
 	\]
 	so from the fact that a matrix and its transpose have the same singular values, we get
 	\begin{equation*}
 	\va_0\|x\|\leq \Bl\|\sum_{j=1}^{d} x_j \xi_j\Br\|\leq \sqrt{d} \|x\|,\   \  \forall x=(x_1,\dots, x_{d})\in\R^{d}.
 	\end{equation*}
 	In particular, this implies that the parallelepiped
 	$$H:=\Bl\{ \sum_{j=1}^{d} x_j \xi_j:\   \  x=(x_1,\dots, x_{d})\in \Bl[-\f 1 {d},\f 1{d}\Br]^{d}\Br\},$$
 	whose edge directions lie in the set $\mathcal{E}$,  satisfies
 	$ B_{\va_0/d}[0]\subset H\subset B_1[0].$
 	Thus, applying Lemma~\ref{lem-3-2-0} to the pair of sets  $(S_0, E)=(H, B)$  and the ball  $B=B_1(0)$,  we obtain
 	$$w_r(B;\CE)_p =w_r(H\cup B; \CE)_p \leq C(p,r,d,\va_0) w_r(H;\CE)_p.$$
 	By  Lemma~\ref{lem-5-3-0}, this implies that $w_r(B; \CE)_p\leq C(p, d,r,\va_0)<\infty$.

 	Next,  we set  $\va:=\f {\va_0^2} {4r}$,  and  apply Lemma ~\ref{lem-6-3}   to write  the domain   $G$  as a finite union $G=\bigcup_{j=1}^m E_j$ such that  $m\leq C(d,r,\va_0)<\infty$, $B_j \subset E_j\subset (1+\va) B_j$ for an open ball $B_j$ of radius $\da>0$,
 	and $E_j\cap E_{j+1}$ contains an open ball of radius $ \va \da/4$.
 	Setting $E_0:=B_1$, and  applying   Lemma~\ref{lem-3-2-0} iteratively to the pairs of sets  $$(S_0, E):=\Bl(\bigcup_{i=0}^j E_i, E_{j+1}\Br)\   \ \text{ for}\  \ j=0,1,\dots, m-1,$$ we obtain
 	\begin{equation*}
 w_r\Bl(\bigcup_{i=0}^{j+1}  E_i;\CE\Br)_p \leq C(d,p,r,\va_0) \Bl( 1+w_r\Bl(\bigcup_{i=0}^j E_i; \CE\Br)_p\Br),\  \ j=0,1,\dots, m-1.
 \end{equation*}
 This implies that\begin{align*}
w_r(G;\CE)_p& =w_r\Bl(\bigcup_{i=0}^{m}  E_i;\CE\Br)_p\leq C(d,p,r,\va_0,L)\Bl( 1+w_r(B_1; \CE)_p\Br)\\
&\leq C(d,p,r,\va_0,L)<\infty.
 \end{align*}
 	
 \end{proof}

 \subsection{Proof of Lemma~\ref{lem-3-2-0}}\label{subsec-6-3}

 \begin{proof} 	Without loss of generality, we assume that $w_r(S_0; \CE)_p<\infty$ (otherwise there's nothing to prove).
 Let $\CE=\{\xi_1,\dots, \xi_d\}$, and let $\xi_{d+j}=-\xi_j$ for $1\leq j\leq d$. Then
 	\begin{equation}\label{6-6a}
 	\sph:=\bigcup_{j=1}^{2d}\{ \xi\in\sph:\  \ \xi\cdot \xi_j\ge \va_0\}.
 	\end{equation}
 	We break the proof into several steps.

 	In the first step,  we prove that   for each pair $(S_0, E)$ of measurable  sets satisfying  $B_0\subset E\subset \sa B_0$ and $B_0\subset S_0$,
 	\begin{equation}\label{6-7a}
 	w_r( S_0\cup E; \CE)_p \leq C(p, r, \va_0,d)\Bl( w_r(S_0; \CE)_p+1\Br),
 	\end{equation}
 	which particularly implies \eqref{6-5-0}   for the special case of  $B=B_0$.
 	
 	Without loss of generality, we may assume that  $B_0=B_1(0)$.
 	Using \eqref{6-6a},    we may decompose the ball  $\sa B_0$ as
 	$\sa B_0  =\bigcup_{j=1}^{2d} A_j$, where
 	\begin{equation*}\label{key-13-5}
 	A_j:=\Bl\{x\in \RR^{d}:    \|x\|<\sa,  \      \  x \cdot \xi_j\ge \va_0\|x\|\Br\}.
 	\end{equation*}
 	We  claim that for each $1\leq j\leq 2d$ and    $h_j:= \f {\va_0}{2r}  \xi_j$,
 	\begin{equation}\label{3-3-0}
 	\bigcup_{x\in A_j} [x-rh_j, x-h_j]\subset B_0.
 	\end{equation}
 	Indeed,   each   $y\in 	[x-rh_j, x-h_j]$  can be written   in the form
 	$y= x  -u \xi_j$ with   $ \f {\va_0} {2r}\leq u\leq \f {\va_0} 2$. Since $x\in A_j$, we have
 	\begin{align*}
 	\|y\|^2& \leq \|x\|^2 +u^2 -2\|x\| u\va_0
 	\leq (\sa- u\va_0)^2 +u^2 < 1-\f {u\va_0}2 +u^2 \leq 1,
 	\end{align*}
 	proving the claim \eqref{3-3-0}.
 	
 	Since $E\subset \sa B_0=\bigcup_{j=1}^{2d} A_j$, we may decompose $E\cup S_0$ as $$E\cup S_0 =\bigcup_{j=0}^{2d} E_j,\   \ \text{where}\  \   E_0:=S_0,\   \ E_j :=E\cap A_j\    \text{for $j\ge 1$}.$$  Let
 	$\Og_k:=\bigcup_{j=0}^k E_j$ for $0\leq k\leq 2d$.
 	Using the claim \eqref{3-3-0}, we obtain that for any $1\leq k\leq m$,
 	$$\bigcup_{x\in E_k} [x-rh_k, x-h_k] \subset B_0\subset S_0\subset \Og_{k-1},$$
 	which, in particular, implies that the condition \eqref{7-2-0-18} is satisfied.  Thus, using Lemma ~\ref{rem:whitney key tool seq appl} and \eqref{3-5-0}, we deduce the desired estimate
 	\eqref{6-7a}.

 	In the second step, we prove that for any ball $B_0\subset S_0$ and  constant  $a\ge 1$,
 	\begin{equation}\label{6-9-0}
 	w_r(S_0\cup a B_0;\CE)_p \leq C(p, r, d, \va_0, a) \bigl(w_r( S_0; \CE)_p+1\bigr).
 	\end{equation}
 	To this end, let $\ell$ be a nonnegative integer such that $\sa^\ell \leq a <\sa^{\ell+1}$. Let
 	$$S_{j,0}:= S_0 \cup \sa^{j-1} B_0,\   \ E_j:=\sa^j B_0,\   \ j=1,2,\dots, \ell,$$
 	and let
 	$$S_{\ell+1, 0}: =S_0\cup \sa^\ell B_0,    \ E_{\ell+1} := a B_0.$$
 	Note that for $1\leq j\leq \ell+1$,
 	$$ \sa^{j-1} B_0 \subset E_j \subset \sa^j B_0 \  \text{and}\   \ \sa^{j-1}B_0 \subset S_{j,0}.$$
 	Thus, applying     \eqref{6-7a} iteratively to the ball  $\sa^{j-1} B_0$ and the pair of sets $(S_{j,0},   E_j)$
 	for $j=1,2,\dots, \ell+1$,  we   obtain
 	\begin{equation*}
 	w_r(S_{j,0}\cup E_j;  \CE)_p \leq C(p, r, \va_0,d) \Bl( w_r(S_{j,0}; \CE)_p+1\Br),\   \ j=1,2,\dots, \ell+1.
 	\end{equation*}
 	Since $S_{1,0}=S_0$, the desired estimate  \eqref{6-9-0} then follows.

 	Finally, in the last step, we prove \eqref{6-5-0} in the general case, where $B=B_{L\da} (x_0)$ and
 	the center $x_0$  may not be  the same as that of $B_0$. Without loss of generality, we  assume that $B_0=B_{\da}(0)$.
 	
 	Since
 	\[ B\subset E\subset \sa B, \    B\subset S_0\cup B \  \text{ and}\  S_0\cup E =(S_0\cup B)\cup E, \]
 	it follows by the  already proven estimate  \eqref{6-7a} that
 	\[
 	w_r(E\cup S_0; \CE)_p \leq C(p, r, \va_0,d)\Bl( w_r(S_0\cup B; \CE)_p+1\Br).
 	\]
 	Thus, it is enough to show that \begin{equation}\label{6-10a}
 	w_r(S_0\cup B; \CE)_p \leq C(p, r, \va_0,d, L)\Bl( w_r(S_0; \CE)_p+1\Br).
 	\end{equation}
 	We consider
 	the following two cases:\\

 	{\bf Case 1.}\   \  $\|x_0\|\leq \f \da {2}$.\\

 	In this case,    $$\wt B:=B_{\da/2} (x_0) \subset B_\da(0)= B_0\subset S_0.$$
 	Since $B=(2L) \wt B$, \eqref{6-10a} follows from \eqref{6-9-0}.\\

 	{\bf Case 2.}\   \  $\|x_0\|> \f \da {2}$.\\

 	In this case, we will construct a set $K$ such that $S_0\cup \wh B\subset K\subset S_0\cup B$ with
 	$\wh B:=B_{\da/\sa} (x_0)$, and
 	\begin{equation}\label{6-11a}
 	w_r(K; \CE)_p \leq C(p, r, \va_0,d, L)\Bl( w_r(S_0; \CE)_p+1\Br).
 	\end{equation}
 	For the moment, we assume such a set $K$ is constructed and proceed with the proof of  \eqref{6-10a}. Since $(L\sa) \wh{ B} =B$ and $K\cup B=S_0\cup B$, we obtain from \eqref{6-9-0} that
 	\[ w_r(S_0\cup B; \CE)_p=w_r(K\cup (L\sa) \wh B; \CE)_p \leq C(p, r, \va_0,d, L)\bigl( w_r(K; \CE)_p+1\bigr),\]
 	which combined  with \eqref{6-11a} implies  \eqref{6-10a}.

 	It remains to construct the  set $K$.	Since $B_\da(0)\subset  B$,  we have    $\f \da{ 2} <\|x_0\|\leq (L-1)\da$.	
 	Let  $n_0$ be a positive integer $< L\sa /(\sa-1)$  such that
 	$$ n_0 \f{(\sa-1)\da}{\sa} < \|x_0\|\leq (n_0+1) \f {(\sa-1) \da}{\sa}.$$
 	Let
 	$$  y_j =j\f{(\sa-1)\da}{\sa} \f {x_0}{ \|x_0\|},\  \ j=0,1,\dots, n_0$$ be the equally spaced points on  the line segment $[0, x_0]$.
 	Define  $\wh B_j :=B_{\da/\sa}(y_j)$ for $j=0,1,\dots, n_0$.  A straightforward calculation shows that     \begin{equation}\label{6-12a}
 	B_{\da/\sa} (x_0)\subset \sa \wh B_{n_0} \subset  B\   \ \text{and}\  \ \wh B_j\subset \sa \wh B_{j-1} \subset B\   \ \text{for $j=1,\dots, n_0$}.
 	\end{equation}
 	Define  $K_0=S_0$ and
 	$K_j:=K_{j-1}\cup \sa \wh B_j$ for $j=1,\dots, n_0$.  	
 	Since  $\sa \wh B_0 =B_0\subset K_0$ and
 	$\wh B_j\subset \sa \wh B_{j-1}\subset K_{j-1}$ for $1\leq j\leq n_0$, it follows by \eqref{6-10a}  that
 	\[ w_r(K_j;\CE)_p=w_r(K_{j-1} \cup \sa \wh B_j;\CE)_p \leq C(p,r,d,\va_0) \Bl(w_r(K_{j-1}; \CE)_p+1\Br),\]
 	where $j=1,2,\dots, n_0$. This implies that
 	\[ 	w_r(K_{n_0};\CE)_p\leq C(p,r,d,\va_0, L) \Bl( w_r(S_{0}; \CE)_p+1\Br),\]
 	and \eqref{6-11a} is satisfied with  $K:=K_{n_0}$. Furthermore, using \eqref{6-12a}, we have
 	$$S_0\subset K_{n_0}=\bigcup_{j=0}^{n_0} (S_0\cup \sa \wh B_j) \subset S_0\cup B\    \ \text{and}\   \ B_{\da/\sa} (x_0) \subset \sa \wh B_{n_0}\subset  K_{n_0}. $$
 	Thus, the set $K=K_{n_0}$ has all the desired  properties.

 \end{proof}

\section{Connection with X-ray numbers of convex bodies}
\label{sec:7}

Let $d\ge 2$. For each convex body $G\subset \R^d$,  we define $\cN_d(G)$ to be the smallest number $n\in\NN$ for which there exists a set  of $n$ directions $\CE$ such that $\spn(\CE)=\RR^d$, and  $ w_r(G; \CE)_p<\infty$ for all $r\in\NN$ and $p>0$.  Clearly,    $\cN_d(G)\ge d$ for every convex body $G\subset\RR^d$. Moreover,  by  Theorem ~\ref{cor-7-8-18},  if the convex body $G\subset \RR^d$ is Lip-2, then
 $\cN_d(G)=d$.

Next, we   define  $\cN_d$ to be the smallest number  $n\in\NN$   such that for   every  convex body $G\subset\RR^d$, there exists a set of $n$ directions $\CE$ for which  $\spn(\CE)=\RR^d$ and
 $w_r(G;\CE)_p<\infty$    for all  $0<p\leq \infty$ and  $r\in\NN$.
 According to   Theorem ~\ref{thm:enlarged set of directions}, we have \begin{equation}
\cN_d=\sup_G \cN_d(G),
\end{equation}
where the supremum is taken over all convex bodies $G\subset \RR^d$.
Moreover, by Remark ~\ref{rem-4-3} and John's theorem for convex bodies,  we have
\begin{equation}\label{7-2}
d\leq \cN_d \leq (cd)^{2(d-1)}.
\end{equation}

The upper estimate in \eqref{7-2} is far from being optimal, especially  as $d\to \infty$.
In this section, we will show that such an   estimate can be significantly improved, proving  that $\cN_2=\cN_2(G)=2$ for every planar convex body $G\subset \RR^2$, and $\cN_d(G)=d$ for $d\ge 3$ and every  ``almost smooth'' convex body $G\subset \RR^d$.  A crucial ingredient in our approach for $d\ge 3$ is to   establish a connection between   the number $\cN_d$ with  the    X-raying numbers of  convex bodies from convex geometry.

We start with   the following result for $d=2$, which implies that $\cN_2=\cN_2(G)=2$ for every planar convex body $G\subset \RR^2$.

\begin{thm}\label{cor-7-7}
	If $G$ is a convex body in $\R^2$, then there exist two linearly independent vectors $\xi_1, \xi_2\in\SS^1$  such that for all $0<p\leq \infty$ and  $r\in\NN$,
	 	\begin{equation}\label{key-13-11}
	 	w_r(G; \{\xi_1,\xi_2\})_p \leq C(p,r)<\infty,
	\end{equation}
	where $C(p,r)>0$ is a constant depending only on $p$ and $r$.
\end{thm}

We will present two short proofs. 

\begin{proof}[Proof A]
	A geometric result of Besicovitch~\cite{Be} asserts that it is possible to inscribe an affine image of a regular hexagon into any planar convex body. Since~\eqref{key-13-11} does not change under affine transform, we may assume the hexagon with the vertices $(\pm1,\pm1)$ and $(0,\pm2)$ (which are the vertices of a regular hexagon after a proper dilation along one of the coordinate axes) is inscribed into $G$, i.e. each vertex of the hexagon is on the boundary of $G$. Convexity of $G$ then implies that $G$ contains the hexagon and is contained in a non-convex $12$-gon (resembling the outline of the star of David) obtained by extending the sides of the hexagon, i.e. having additional nodes $(\pm2,0)$ and $(\pm1,\pm3)$. Choosing $\xi_1=(1,0)$ and $\xi_2=(0,1)$, the proof is now completed by two applications of Lemma~\ref{cor-7-3} for $S=[-1,1]^2$: first in the direction of $\xi_1$ with $L=2$, and second  in the direction of $\xi_2$ with $L=3$.
\end{proof}
\begin{proof}[Proof B]
	By John's theorem,  without loss of generality we may assume that  $B_1[0]\subset G\subset B_2[0]$. Let $a, b\in G$ be such that $\|a-b\|=\text{diam} (G)=L\leq 4$. Let $\xi_1$ be the unit vector in the direction of $b-a$ and $\xi_2$ the unit vector that is perpendicular to $\xi_1$.  Without loss of generality, we may assume that $\xi_1=(0,1)$ and  $\xi_2=(1,0)$, since otherwise we may consider a  rotation of $G$. Then 
	$a=(0, -L_1)$ and $b=(0, L_2)$ for some $L_1, L_2\ge 1$ such that $L_1+L_2=L$, and   $S:=[-\f 1{\sqrt{2}}, \f 1{\sqrt{2}}]^2\subset B_1[0]\subset G$.  Let    $G_1:=( [ -\f 1{\sqrt{2}}, \f 1{\sqrt{2}}]\times [-L_1, L_2])\cap G$.  Since $G$ is convex, the set $G_1$ satisfies the conditions of
	Corollary \ref{cor-7-3} in  the direction of $e_2$. Thus, using   Theorem  \ref{cor-7-3} and Remark \ref{rem-7-4}, 
	we get 
	$$E_{r-1}(f)_{L^p(G_1)} \leq C(p,r) \og^r(f, G_1;  \{e_1, e_2\})_p.$$	
	On the other hand, 
	since $\diam(G)=\|a-b\|=L_1+L_2$, it follows that  for each $(x, y)\in G$,   $-L_1\leq y \leq L_2$, which implies that   $(0,y)\in G_1$, and $[(0,y), (x, y)]\subset G$.  Thus, applying
	Remark \ref{rem-7-4} in the direction of $e_1$ with $K=G_1$, we obtain  the desired inequality \eqref{key-13-11}. 
\end{proof}

\begin{rem}
	Proof~A is a bit shorter and provides explicit small values of the parameter $L$ from Lemma~\ref{cor-7-3} which would result in a better value of the constant in~\eqref{key-13-11}. We are unaware of any algorithms for construction of an inscribed affine-regular hexagon into given planar convex body, although the variational proof from~\cite{Iu}*{Th.~2.3} may be a basis for such an algorithm. Proof~B uses John's theorem, i.e. the largest area ellipse inscribed into the convex body. This ellipse can be constructed algorithmically, see, e.g.~\cite{Co}, \cite{Da}, \cite{Su} and references therein.
\end{rem}

The problem of finding the exact value of the number $\cN_d$ for $d\ge 3$  appears to be more challenging. The proofs of Theorem ~\ref{cor-7-7} above cannot be immediately extended to  the higher-dimensional convex bodies.

 For the reminder of this section, we assume that $d\ge 3$.  Our aim  is to  relate  the directional Whitney inequality to  the concept of X-raying from convex geometry, which allows us to deduce improved  upper estimates of  $\cN_d$   for $d\ge 3$ from  certain  known results on X-ray  numbers of convex bodies.

 We need to introduce some definitions from convex geometry, which can be found in \cite{Be-survey} and \cite{Be-Xray}.
 Let $G$ be a convex body in $\RR^d$.  We say a point $x$ on the boundary   of  $G$ is illuminated in   the direction ${e}\in\SS^{d-1}$ if the ray $\{x+t{e}:t\ge0\}$ intersects the interior  of $G$.
 A set of directions $\EEE\subset \SS^{d-1}$ is said to illuminate $G$ if every boundary  point of $G$ is illuminated in a direction from the set $\CE$.  One can consider illumination of $G\subset\RR^d$  by parallel beams of light.  The illumination number $I_d(G)$ of $G$  is the smallest $n\in\NN$  for which $G$ can be illuminated by a set of $n$ directions.
Next, we say  $G$ is X-rayed by a set of directions  $\EEE\subset \SS^{d-1}$ if $G$ is illuminated by the set  $\EEE\cup(-\EEE)$, and define the  X-ray number $X_d(G)$ of $G$  to be the smallest number $n\in\NN$ for which $G$ can be  X-rayed by a set of $n$ directions. By the definition, we have \begin{equation}\label{7-4}
X_d(G)\leq I_d(G) \leq 2 X_d(G).
\end{equation}
Finally, for $d\ge 3$,  we say  $G\subset\RR^d$ is   almost smooth if for each  boundary point  $x\in\p G$, and  any  two outer unit normal vectors $u_1$, $u_2$ of the supporting hyperplane(s) of $G$ at $x$,   the inequality   $u_1\cdot u_2\ge (d-2)/(d-1)$ holds.

In this section, we will prove the following result:

 \begin{thm}\label{thm:xray}
 	If  $G\subset\R^d$ is a convex body in $\R^d$ X-rayed by a finite set of directions $\EEE\subset \SS^{d-1}$, then  for any $r\in\NN$ and $0<p\leq \infty$,
 	\begin{equation}\label{eqn:ineq for enlarged}
 	w_{r}  (G;\EEE)_{p} \leq C_{p, r, G}<\infty.
 	\end{equation}
 	In particular, this implies that for every convex body $G\subset \RR^d$,
 	\begin{equation}\label{7-6}
 	\cN_d(G)\le  X_d(G).
 	\end{equation}
 \end{thm}

 We postpone the proof of Theorem ~\ref{thm:xray} until Subsection \ref{subsec-7-1}.    For the moment, we
  take it  for granted and prove  a few useful corollaries for $d\ge 3$.
 First, it was shown in \cite{Be-Xray}*{Th.~2.3} that $X_d(G)=d$ for  every  almost smooth convex body $G\subset\RR^d$. This combined with  Theorem~\ref{thm:xray}  implies

 \begin{cor}\label{thm:C2 all dirs strong}
 	If $d\ge 3$ and $G$ is an  almost smooth convex body  in $\RR^d$, then  $\cN_d(G)=d$.
 \end{cor}

 Next, we define
 $$X_d:=\sup_{G} X_d(G),$$
 where the supremum is taken over all convex bodies $G$ in $\RR^d$.
 For the maximal  X-ray number $X_d$, it is conjectured and  proved only for $d=2$ that $X_d\le 3\cdot 2^{d-2}$
  (see~\cite{Be-Xray} and references therein).
  A better studied  concept is the illumination number, which is related to the  X-ray number  via \eqref{7-4}.   The well-known illumination conjecture in convex geometry states (\cite{Be-survey})   that $$I_d:=\sup_{G} I(G)\le 2^d,$$
  where  the supremum is  taken over all convex bodies $G\subset \RR^d$.
    Again, this conjecture  was proved for $d=2$ only (see~\cite{Be-survey}). Let us provide a summary of   known upper bounds for  $I_d$:
    \begin{itemize}
    	\item  By~\cite{Pa}, $I_3\le 16$.
    \item 	The estimates $I_4\le 96$, $I_5\le 1091$ and $I_6\le 15373$ were recently obtained in~\cite{Pr-Sh}.
    \item  For  $d\ge 7$, the following  explicit bound follows from the results of Rogers and Shepard ~\cites{Ro,Ro-Sh} (see also~\cite{Be}*{Section~2.2}):
     $$I_d\le \binom{2d}{d}d(\ln d+\ln\ln d+5),$$
     where $5$ can be replaced with $4$ for sufficiently large $d$.
    \item The following  remarkable asymptotic estimate of $I_d$  with an implicit  constant $c_0>0$
      was obtained  very recently in~\cite{HSTV}:
    $$I_d\le\binom{2d}{d} e^{-c_0\sqrt{d}}.$$
    \end{itemize}

     Using \eqref{7-6}, \eqref{7-4} and these  known bounds for $I_d$, we obtain
 \begin{cor}
 For $d\ge 3$, we have
 	\begin{equation*}
 	\cN_d \le \begin{cases}
 	16, & \text{if }d=3,\\
 	96, & \text{if }d=4,\\
 	1091, & \text{if }d=5,\\
 	15373, & \text{if }d=6,\\
 	\binom{2d}{d} \min\{d(\ln d+\ln\ln d+5),e^{-c_0\sqrt{d}}\}, & \text{if }d\ge7.
 	\end{cases}
 	\end{equation*}
 \end{cor}

%


\subsection{Proof of Theorem ~\ref{thm:xray}}\label{subsec-7-1}

\begin{proof}Without loss of generality, we may assume that $\EEE=-\EEE$ and  $0\in G^\circ$, where $G^\circ$ denotes the interior of $G$.  For a direction ${e}\in\SS^{d-1}$, we define
	\begin{align*}
	R(S,{e}):&=\{x-t{e}:x\in S,\,t\ge0 \}\   \ \text{for any  $S\subset G$},\\
	\vi(x,e):&=\max\{ t\ge 0:\  \ x-te\in G\}\  \ \text{for any  $x\in G$}.
	\end{align*}
	Clearly, $0\leq \vi(x,e)\leq \diam(G)$, and
	$$R(S,e)\cap G=\{x-te:\  \ x\in S,\  \ 0\leq t\leq \vi(x,e)\}\   \ \text{for any $S\subset G$}.$$

We shall use the following known result from convex geometry  (see  ~\cite{Ha}):  there exists   an increasing  sequence $\{S_n\}_{n=1}^\infty$ of open, convex and $C^2$-subsets of $G$  such that $G^\circ=\bigcup_{n= 1}^\infty S_n$,
	 \begin{equation}\label{7-3}
	 \Bl(1-\f 1 {n+1}\Br) G \subset S_n \subset \Bl(1-\f 1{n+2}\Br) G,\   \ n=1,2,\dots.
	 \end{equation}
	   (Indeed, one can choose each $S_n$  to be an algebraic domain.)	
We break the rest of the proof into several steps.

First, we claim that there exists a positive integer  $n_0$ such that
	\begin{equation}\label{eqn:RS cover}
	G\subset \bigcup_{{e}\in\EEE} R(S_{n_0},{e}).
	\end{equation}
	Assume otherwise  that for each $n$,
	\begin{equation*}
	T_n:=G\setminus \left(\bigcup_{{e}\in\EEE} R(S_n,{e})\right)\neq \emptyset.
	\end{equation*}
	 Since  $\{T_n\}_{n\ge 1}$ is a decreasing sequence of closed subsets of $G$,  by Cantor's intersection theorem, there exists a point $x\in \bigcap_{n=1}^\infty T_n$. Since $S_n\subset R(S_n,{e})$ for any ${e}\in \sph$ and $\bigcup_{n= 1}^\infty S_n= G^\circ=G\setminus \partial G$, we have $x\in\partial G$. Thus,  $x$ must be illuminated in a direction ${e}\in\EEE$, which means that there exists $t\ge 0 $ and $e\in\CE$ such that   $x+t{e}\in G^\circ=\bigcup_{n=1}^\infty S_n$. This implies that  $x\in R(S_n,{e})$ for some $n$  and direction $e\in\CE$,
	 which  is impossible as  $$x\in T_n=G\setminus \left(\bigcup_{{e}\in\EEE} R(S_n,{e})\right),\  \ \forall n\in\NN.$$
	 This proves the claim \eqref{eqn:RS cover}.

	Next, we   set $n_1=n_0+2$ and
	$$G_e: =R(S_{n_0}, e) \cap G,\   \ e\in\CS.$$	
	We prove that for each direction  $e\in\CE$ and  measurable set $S_{n_1} \subset S \subset G$,
	\begin{equation}\label{7-5}
	w_r(G_e\cup S;\CE)_p \leq C(p,r,G) \Bl( 1+w_r(S;\CE)_p\Br).
	\end{equation}

	Indeed,  by \eqref{7-3},
	$$ \overline{S_{n_0}} \subset \Bl(1-\f 1 {n_0+2} \Br) G\subset \Bl(1-\f 1 {n_0+3} \Br) G^\circ \subset S_{n_1}.$$
	Thus, 	there exists  $\delta>0$ such that
	\begin{equation}\label{eqn:N_1 def}
	\overline{S_{n_0}}+B_{r\delta}[0]\subset S_{n_1}.
	\end{equation}
	Now  we define
	\begin{align*}
	K_{e,0}:&=\{ x-te:\  \ x\in S_{n_0},\  \ 0\leq t\leq r\da\},\\
	 K_{e,j}:&=\Bl\{x-te:\  \ x\in S_{n_0},\
	  (r+j-1)\da \leq t \leq \min\{ (r+j) \da,  \vi(x,e)\}\Br\},\\
	  &  \hspace{5cm} \  \ j=1,2,\dots, \ell_e,
	\end{align*}
	where $\ell_e\leq \f {\diam \ G} \da$ is the largest integer such that $(r+\ell_e-1)\da \leq \vi(x,e)$ for some $x\in S_{n_0}$.  By \eqref{eqn:N_1 def}, $K_{e,0}\subset S_{n_1}\subset G$, and thus, for  $S_{n_1}\subset S\subset G$,
	$$G_e\cup S=S \cup \Bl(\bigcup_{j=1}^{\ell_e} K_{e, j}\Br).$$
Furthermore, for   $h=\da e$ and each $1\leq j\leq \ell_e$, we have
	 \begin{align*}
	 \bigcup_{i=1}^r ( K_{e,j} -ih)\subset S \cup \Bl( \bigcup_{v=1}^{j-1} K_{e,v}\Br).
	 \end{align*}
	The estimate \eqref{7-5} then follows by 	 Lemma  \ref{rem:whitney key tool seq appl}.
	
	Third, we show that \begin{equation}\label{7-11}
	w_r(G;\CE)_p \leq C \Bl( 1+w_r(\overline{S_{n_1}};\CE)_p\Br).
	\end{equation}
	
	Let $\CE=\{\xi_j\}_{j=1}^m$. By \eqref{eqn:RS cover}, we have
		$$ G=\bigcup_{j=1}^m \Bl( R(S_{n_0}, \xi_j) \cap G\Br)=\bigcup_{j=0}^m G_j,$$
	where  $$G_0=\overline{S_{n_1}}\  \ \text{ and}\  \ G_j ={G_{\xi_j}} \cup S_{n_1}\  \ \text{ for}\  \ 1\leq j\leq m. $$
Using  \eqref{7-5} with $S=\bigcup_{j=0}^{k-1} G_j$  iteratively for $k=1,2,\dots, m $, we obtain
\begin{align*}
w_r\Bl( \bigcup_{j=0}^k G_j;\CE)_p \leq C \Bl( 1+ w_r \bigl( \bigcup_{j=0}^{k-1} G_j;\CE)_p\Br),\  \ k=1,2,\dots, m.
\end{align*}
Thus,
\begin{align*}
w_r(G;\CE)_p =w_r\Bl( \bigcup_{j=0}^m G_j; \CE\Br)_p \leq C \Bl( 1+w_r(\overline{S_{n_1}};\CE)_p\Br).
\end{align*}	
	
Finally, we show that
$$	w_r(\overline{S_{n_1}};\CE)_p\leq C_{p,r,G}<\infty,$$
which combined with \eqref{7-11}  will imply the desired estimate $w_r(G;\CE)_p<\infty$.
Since $\overline{S_{n_1}}$ is a convex $C^2$-domain, by Theorem ~\ref{cor-7-8-18} and Theorem  ~\ref{thm:enlarged set of directions}, it suffices to show that $\spn(\CE)=\RR^d$.

Assume otherwise. Then  there exists a direction $\eta\in\SS^{d-1}$ orthogonal to any vector in $\EEE$. Let $H$ be one of the two supporting hyperplanes to $G$ orthogonal to $\eta$. Then any point of $H\cap G$ cannot be illuminated by any direction of $\EEE$, obtaining a contradiction.

This completes the proof of Theorem ~\ref{thm:xray}.
\end{proof}

{\bf Acknowledgment.} The authors are grateful to the anonymous referee for the comments that improved the paper.

\end{document}